\documentclass[12pt]{amsart}
\usepackage{amsfonts,amssymb,amscd,amstext,mathrsfs}
\usepackage[utf8]{inputenc}
\usepackage{hyperref}
\usepackage{verbatim}

\usepackage{graphics}
\usepackage{graphicx}

\usepackage{times}
\usepackage{enumerate}
\usepackage[up,bf]{caption}
\usepackage{color}

\input xy
\xyoption{all}

\newtheorem{theorem}{Theorem}[section]
\newtheorem{proposition}[theorem]{Proposition}

\newtheorem{lemma}[theorem]{Lemma}
\newtheorem{corollary}[theorem]{Corollary}

\theoremstyle{definition}
\newtheorem{definition}[theorem]{Definition}
\newtheorem{remark}[theorem]{Remark}

\numberwithin{equation}{section}
\numberwithin{figure}{section}

\newcommand\Fcal{\mathcal{F}}
\newcommand\Gcal{\mathcal{G}}

\newcommand\Oscr{\mathscr{O}}

\newcommand\B{\mathbb{B}}
\newcommand\C{\mathbb{C}}

\newcommand\N{\mathbb{N}}

\newcommand\ggot{\mathfrak{g}}
\newcommand\hgot{\mathfrak{h}}
\newcommand\igot{\mathfrak{i}}

\renewcommand\igot{\mathfrak{i}}

\renewcommand\imath{\igot}

\newcommand\dist{\mathrm{dist}}

\newcommand\Aut{\mathrm{Aut}}

\newcommand\Id{\mathrm{Id}}

\def\dist{\mathrm{dist}}

\usepackage[color=blue!20]{todonotes}

\begin{document}

\title{Complete nonsingular holomorphic foliations \\ on Stein manifolds}

\author{Antonio Alarc\'on \; and \; Franc Forstneri{\v c}}

\address{Antonio Alarc\'on, Departamento de Geometr\'{\i}a y Topolog\'{\i}a e Instituto de Matem\'aticas (IMAG), Universidad de Granada, Campus de Fuentenueva s/n, E--18071 Granada, Spain}
\email{alarcon@ugr.es}

\address{Franc Forstneri\v c, Faculty of Mathematics and Physics, University of Ljubljana, and Institute of Mathematics, Physics, and Mechanics, Jadranska 19, 1000 Ljubljana, Slovenia}
\email{franc.forstneric@fmf.uni-lj.si}

\subjclass[2010]{Primary 32M17, 32M25; secondary 32H02, 37F75}

\date{26 December 2023}

%%%%%

\keywords{Stein manifold; complete holomorphic foliation; density property}

\begin{abstract}
Let $X$ be a Stein manifold of complex dimension $n>1$ endowed with a 
Riemannian metric $\ggot$. We show that for every integer $k$ with 
$\left[\frac{n}{2}\right] \le k \le n-1$ there is a nonsingular holomorphic foliation 
of dimension $k$ on $X$ all of whose leaves are closed and $\ggot$-complete.
The same is true if $1\le k<\left[\frac{n}{2}\right]$ provided that there 
is a complex vector bundle epimorphism $TX\to X\times\C^{n-k}$.
We also show that if $\Fcal$ is a proper holomorphic foliation on $\C^n$
$(n>1)$ then for any Riemannian metric $\ggot$ on $\C^n$ there is a holomorphic
automorphism $\Phi$ of $\C^n$ such that the image foliation 
$\Phi_*\Fcal$ is $\ggot$-complete. The analogous result is obtained
on every Stein manifold with Varolin's density property.
\end{abstract}

\maketitle

\centerline{\em Dedicated to Josip Globevnik}
\bigskip

%
%
%    INTRODUCTION
%
%
\section{Introduction}\label{sec:intro} 

Let $X$ be a complex manifold endowed with 
% an arbitrary 
% Antonio: I added "arbitrary", to make the referee happy.
% Franc: this looks superfluous. I doubt that the referee will complain the second time.
a Riemannian metric $\ggot$.
A locally closed subvariety $Y$ of $X$ is said to be $\ggot$-complete 
if every smooth divergent path $\gamma:[0,1)\to Y$ 
(i.e., one leaving every compact subset of $Y$) has infinite $\ggot$-length: 
$
	\int_0^1 \ggot(\gamma(t),\dot\gamma(t)) \,dt=+\infty.
$
If $\ggot$ is a complete metric on $X$ then clearly every closed connected complex 
subvariety of $X$ of positive dimension is $\ggot$-complete. 
A holomorphic foliation on $X$ is said to be $\ggot$-complete if every 
leaf is $\ggot$-complete. 
For the theory of holomorphic foliations, see Sc{\'a}rdua \cite{Scardua2021}.

The results of this paper concern the problem, posed by 
Yang in 1977 \cite{Yang1977JDG,Yang1977AMS}, whether there exist 
bounded complex submanifolds of a Euclidean space $\C^n$ for $n>1$
which are complete in the Euclidean metric.
We begin with a brief overview of the main developments 
on this subject, referring the interested reader to the more complete recent 
survey by Alarc\'on \cite{Alarcon2022Yang}. 
We also point out that Yang's problem is related to the Calabi--Yau problem 
in the theory of minimal surfaces; see \cite[Chapter 7]{AlarconForstnericLopez2021} 
for a recent survey of the latter subject.

The first affirmative result on Yang's problem was obtained 
by Jones \cite{Jones1979}, who constructed a holomorphic immersion of the unit 
disc into $\C^2$ and an embedding into $\C^3$ with bounded image 
and complete induced metric. His method is based on the BMO duality theorem. 
Much later, Mart\'{i}n, Umehara, and Yamada \cite{MartinUmeharaYamada2009PAMS} 
used a geometric method to construct complete 
bounded complex curves in $\C^2$ with arbitrary finite genus and finitely many ends.
This was followed by Alarc\'{o}n and L\'{o}pez \cite{AlarconLopez2013MA} who 
showed that any topological type is possible. 
The methods used in these constructions do not 
provide any control of the complex structure on the curve, except in the 
simply connected case when any such curve is biholomorphic to the disc. 
In \cite{AlarconForstneric2013MA}, the authors used the Riemann--Hilbert boundary 
value problem and gluing methods to prove that every bordered Riemann surface 
admits a complete proper holomorphic immersion into the ball of $\C^2$  
and a complete proper holomorphic embedding into the ball of $\C^3$. 
A related result for Riemann surfaces with finite genus and countably many ends
was obtained by the authors in \cite[Theorem 1.8]{AlarconForstneric2021RMI}. 
%(The focus there is on minimal surfaces, but the same method applies to holomorphic curves.)
These are still the only results on Yang's problem with a complete control of the conformal structure on the underlying Riemann surface.

All results mentioned so far pertain to complex curves. 
This is how things stood until the seminal work of
Globevnik \cite{Globevnik2015AM} in 2015. In his landmark construction,
Globevnik showed that for every pair of integers $1\le k<n$ there exists 
a holomorphic foliation of the unit ball $\B^n$ in $\C^n$ 
by closed complete complex subvarieties of complex dimension $k$, 
most of which are smooth (without singularities). 
The leaves of foliations in his construction are connected 
components of the level sets of holomorphic maps  $f:\B^n \to \C^{n-k}$. 
Completeness of the leaves is ensured by choosing the map $f$ to grow 
sufficiently fast on components of a 
suitable labyrinth $\Gamma\subset\B^n$ having the property that any divergent path
in $\B^n$ avoiding all but finitely many components of $\Gamma$
has infinite Euclidean length (see Lemma \ref{lem:labyrinth}). 
The construction of such labyrinths was one of the main new results
of Globevnik's paper. In the sequel \cite{Globevnik2016MA}, Globevnik
extended his construction to any pseudoconvex domain in $\C^n$. 
Further improvements and generalizations of his results were made 
by Alarc\'on \cite{Alarcon2022JDG,Alarcon2022IUMJ}.

This is a suitable point to state our first main result. It generalizes Globevnik's theorem
to an arbitrary Riemannian Stein manifold, and the foliations that we find are nonsingular.
% Antonio: does it worth to mention the following, for emphasis on the news?
% Franc: I think it is superfluous and may look ridiculous, every complex analyst knows this.
% Recall that a domain in $\C^n$ is pseudoconvex if and only if it is a Stein manifold; not every $n$-dimensional Stein manifold is biholomorphic to a pseudoconvex domain in $\C^n$, though.

%
% MAIN THEOREM
%
\begin{theorem}\label{th:main}
Let $X$ be a Stein manifold of complex dimension $n>1$ endowed with a 
Riemannian metric $\ggot$. For every integer $k$ with 
$\left[\frac{n}{2}\right] \le k \le n-1$ there exists a nonsingular holomorphic foliation 
of dimension $k$ on $X$ all of whose leaves are closed and $\ggot$-complete.
The same holds if $1\le k<\left[\frac{n}{2}\right]$ provided that there 
is a complex vector bundle epimorphism $TX\to X\times\C^{n-k}$. 
In particular, if $X$ is parallelizable then it admits a nonsingular 
holomorphic foliation of any dimension $k\in\{1,\ldots,n-1\}$ with 
closed $\ggot$-complete leaves.
\end{theorem}

Foliations in Theorem \ref{th:main} are given by holomorphic submersions 
$X\to \C^{n-k}$. The proof (see Section \ref{sec:proof1}) combines the methods 
of Globevnik \cite{Globevnik2015AM,Globevnik2016MA} with those of Forstneri\v c 
\cite{Forstneric2003AM}; the latter provide the h-principle 
for holomorphic submersions from any Stein manifold $X$
to Euclidean spaces of dimension $<\dim X$. 

%
%  THE SECOND TOPIC
%
The construction methods used in the proof of Theorem \ref{th:main}
(like those in 
\cite{Globevnik2015AM,Globevnik2016MA,Alarcon2022JDG,Alarcon2022IUMJ}) 
do not provide any control of the topology or the complex structure of the leaves. 
Soon after Globevnik's work, Alarc\'on, Globevnik, and L\'opez 
\cite{AlarconGlobevnik2017,AlarconGlobevnikLopez2019Crelle}
combined the use of labyrinths with the approximation theory 
for holomorphic automorphisms of 
complex Euclidean spaces to construct complete properly embedded complex 
hypersurfaces in the ball $\B^n$ of $\C^n$ $(n>1)$ with partial control on their 
topology, and with prescribed topology if $n=2$. The main idea 
is to inductively deform a given closed complex submanifold $Y\subset\C^n$ 
by holomorphic automorphisms of $\C^n$ such that the images avoid more
and more pieces of a given labyrinth $\Gamma\subset\B^n$, and they converge 
to a properly embedded complete complex submanifold of $\B^n$ which is 
biholomorphic to a Runge domain in $Y$. This construction is possible if 
$\Gamma$ is polynomially convex and the connected components of $\Gamma$ 
are holomorphically contractible. The labyrinths used in the aforementioned papers 
consist of balls in suitably placed affine real hyperplanes in $\C^n$. 
A bit later, the authors proved in  \cite{AlarconForstneric2020MZ} that the ball $\B^n$  
for $n>1$ admits a nonsingular holomorphic foliation by complete holomorphic discs. 
By using labyrinths in pseudoconvex shells in $\C^n$, constructed 
by Charpentier and Kosi\'{n}ski in \cite{CharpentierKosinski2020}, one obtains 
the analogous result with the ball replaced by an arbitrary Kobayashi hyperbolic 
pseudoconvex Runge domain in $\C^n$ with $n>1$ 
(see \cite[Remark 1]{AlarconForstneric2020MZ}). If the domain fails to be
hyperbolic then some leaves of the foliation may be complex lines. 
See the survey \cite{Alarcon2022Yang} for more information.

In order to extend this technique to more general Stein manifolds, we must
assume that the manifold has many holomorphic automorphisms.
The suitable class are Stein manifolds with the density property;
see Definition \ref{def:density}. Another technical issue is to 
find suitable labyrinths in $X$ with holomorphically contractible components. 
We do this in Section \ref{sec:proper} in the course of 
proof of Theorem \ref{th:proper2} (see Lemma \ref{lem:labyrinth2},
which is an important new tool). 

We shall consider foliations satisfying the following condition.

%
%
%  PROPER FOLIATIONS
%
\begin{definition}\label{def:proper}
Let $X$ be a connected Stein manifold of dimension $>1$. 
A (possibly singular) holomorphic foliation $\Fcal$ on $X$ 
is {\em proper} if every leaf $\Fcal_x$ $(x\in X)$ is closed and satisfies 
$\dim \Fcal_x\ge 1$, and for every compact subset $K\subset X$ the set 
\begin{equation}\label{eq:FK}
	\Fcal(K) := \overline{\bigcup_{x \in K}\Fcal_x} 
\end{equation} 
is such that $X\setminus \Fcal(K)$ is nonempty and not relatively compact in $X$.
\end{definition}

A biholomorphic map clearly takes a proper foliation to another such foliation. 

\begin{proposition}\label{prop:proper}
Every holomorphic foliation $\Fcal$ on a connected
Stein manifold $X$ given by a nonconstant holomorphic map 
$f=(f_1,\ldots,f_q):X\to\C^q$ with $1\le q<\dim X$ is proper. 
\end{proposition}

\begin{proof}
Every irreducible component of a nonempty fibre $f^{-1}(z)$ $(z\in \C^q)$ is a 
closed complex subvariety of $X$ of some dimension $d\in \{1,\ldots,\dim X-1\}$. 
In particular, the foliation does not have any zero dimensional leaves.
Given a compact set $K\subset X$ we clearly have that 
\begin{equation}\label{eq:inclusion}
	\Fcal(K) \subset f^{-1}(f(K))\subset f_j^{-1}(f_j(K))\ \ 
	\text{for every $j=1,\ldots,q$.}
\end{equation}
Since $f$ is nonconstant, the function $f_j:X\to \C$ is nonconstant for some $j$, 
hence an open map. Thus, $f_j(X)$ is an open subset of $\C$ 
containing the compact subset $f_j(K)$. Choose a nonempty open subset 
$U\subset f_j(X)\setminus f_j(K)$. The set
$X\setminus f_j^{-1}(f_j(K))$ then contains $f_j^{-1}(U)$, hence 
is nonempty and not relatively compact.
By \eqref{eq:inclusion} the same is true for $X\setminus \Fcal(K)$.
\end{proof} 

The proof of Proposition \ref{prop:proper} also gives the following criterium for
properness.

\begin{corollary}\label{cor:proper}
Assume that $X$ is a connected Stein manifold of dimension $>1$ and
$\Fcal$ is a holomorphic foliation on $X$ with all leaves closed and of positive
dimension. If there is a nonconstant holomorphic function on $X$ which is
constant on every leaf of $\Fcal$, then $\Fcal$ is a proper foliation.
\end{corollary}

An example of a nonproper (singular) holomorphic foliation is given 
by punctured complex lines through the origin in $\C^2$, with the origin a
leaf of dimension zero. Taking a Cartesian product with $\C$ gives a 
nonproper holomorphic foliation of $\C^3$ with a closed leaf of dimension
$1$ and nonclosed leaves of dimension $2$. 
Proposition \ref{prop:proper} and Corollary \ref{cor:proper} show that  there are no 
simple examples of nonproper holomorphic foliations having closed leaves.

Our second main result is the following. It is proved in Section \ref{sec:proper}.

%
%	Making a foliation g-complete by an automorphism
%
\begin{theorem}\label{th:proper1}
Let $\ggot$ be a Riemannian metric on $\C^n$, $n>1$.
For every proper holomorphic foliation $\Fcal$ on $\C^n$ (see Definition \ref{def:proper})
there is a holomorphic automorphism $\Phi\in\Aut(\C^n)$ 
such that the image foliation $\Phi_*\Fcal$ with leaves $\Phi(\Fcal_z)$ $(z\in \C^n)$
is $\ggot$-complete. 
\end{theorem}

This theorem is nontrivial if the metric $\ggot$ is not complete on $\C^n$,
and its main interest is when $\ggot$ decays fast at infinity.
Note that the foliation $\Phi_*\Fcal$ has exactly the same leaves as the original foliation
$\Fcal$ up to biholomorphisms, but they are now sufficiently twisted in $\C^n$
to become $\ggot$-complete.

By using labyrinths provided by Lemma \ref{lem:labyrinth2}, 
we shall also prove the analogue of Theorem \ref{th:proper1} for every Stein manifold 
with the density property, a notion introduced by Varolin \cite{Varolin2001}.

Recall that a holomorphic vector field on a complex manifold $X$ is called complete 
if its flow exists for all complex values of time, so it forms a complex
one-parameter group of holomorphic automorphisms of $X$.

%
%   DENSITY PROPERTY
%
\begin{definition}[See Varolin \cite{Varolin2001} 
or Definition 4.10.1 in \cite{Forstneric2017E}] \label{def:density}
A complex manifold $X$ has the {\em density property} if every holomorphic 
vector field on $X$ can be approximated uniformly on compacts by %Lie combinations
sums and commutators of complete holomorphic vector fields on $X$.
\end{definition}

The fact that the Euclidean space $\C^n$ for $n>1$ has the density property 
was discovered by Anders\'en and Lempert \cite{AndersenLempert1992}, 
thereby giving birth to this theory. It is known that most complex Lie groups 
and complex homogeneous manifolds have the density property.
Surveys of this subject can be found in \cite[Chapter 4]{Forstneric2017E}, 
\cite{ForstnericKutzschebauch2022}, and  \cite{Kutzschebauch2020}. 
An important point is that, on any Stein manifold $X$ with the density property,  
one can approximate isotopies of biholomorphic maps between 
Stein Runge domains in $X$ by isotopies of holomorphic automorphisms of $X$; 
see Forstneri\v c and Rosay \cite{ForstnericRosay1993} for the case $X=\C^n$
and \cite[Theorem 4.10.5]{Forstneric2017E} for the general case. 
This result plays an essential role in the proofs of Theorems 
\ref{th:proper1}, \ref{th:proper2}, and \ref{th:proper3}.

The following result generalizes Theorem \ref{th:proper1}; 
it is proved in Section \ref{sec:proper}. 
As above, it is nontrivial only if the given metric $\ggot$ on $X$ is not complete.

%
%	Making a foliation g-complete by an automorphism
%
\begin{theorem}\label{th:proper2}
Let $X$ be a Stein manifold
%
% Antonio: I added the dimension, or is it that no open Riemann surface has the density property? The same in the next theorem
% Franc: there are no Riemann surfaces with DP, it is well known because
% ever RS has finite dim Aut group.
% of dimension $>1$
with the density property, endowed with a 
Riemannian metric $\ggot$. For every proper holomorphic foliation $\Fcal$ on $X$
there is a holomorphic automorphism $\Phi\in\Aut(X)$ such that the image foliation 
$\Phi_*\Fcal$ is $\ggot$-complete. 
\end{theorem}

The analogous construction applies on any pseudoconvex Runge
domain in a Stein manifold with the density property. This gives the following
result, generalizing the aforementioned results in
\cite{AlarconForstneric2020MZ,AlarconGlobevnik2017,AlarconGlobevnikLopez2019Crelle} 
for domains in $\C^n$ with $n>1$. See also Theorem \ref{th:single}. 

%
%	Making a foliation g-complete by an automorphism
%
\begin{theorem}\label{th:proper3}
Let $X$ be a Stein manifold
of dimension $>1$
 with the density property, and let $\Omega$ be a 
pseudoconvex Runge domain in $X$ endowed with a 
Riemannian metric $\ggot$. For every proper holomorphic foliation $\Fcal_0$ 
on $X$ there is a $\ggot$-complete holomorphic foliation $\Fcal$ on $\Omega$ 
such that every leaf of $\Fcal$ is biholomorphic to a pseudoconvex Runge domain in a
leaf of $\Fcal_0$.  
\end{theorem}

Theorems \ref{th:proper1}, \ref{th:proper2}, and \ref{th:proper3} suggest
that the main obstruction to finding complete 
holomorphic foliations with specific types of leaves on a given 
Stein Riemannian manifold with the density property lies in the topological
and the complex structure of the manifold, and not in the choice of the Riemannian metric.

%%%%%%%%%%%%%%%%%%%%%%
%
%   SECTION: PROOF OF THEOREM 1.1
%
%%%%%%%%%%%%%%%%%%%%%%

\section{Proof of Theorem \ref{th:main}}\label{sec:proof1}
We begin by sketching the proof of Globevnik's main theorem in \cite{Globevnik2015AM}.
The essential ingredient are labyrinths $\Gamma$ as in the following lemma
from \cite{Globevnik2015AM}. A simpler construction 
can be found in \cite{AlarconGlobevnikLopez2019Crelle},
where the connected components of $\Gamma$ are balls in affine real hyperplanes.

%
% LABYRINTH
%
\begin{lemma}\label{lem:labyrinth}
Given numbers $0<r<s$ and $M>0$, there is a compact set 
$\Gamma\subset s\B^n\setminus r\overline \B^n$ with connected complement
satisfying the following two conditions:
\begin{enumerate}[\rm (a)]
\item the compact set $r \overline \B^n \cup \Gamma$ is polynomially convex, and 
\item every piecewise smooth path $\gamma:[0,1]\to \C^n \setminus \Gamma$ 
with $\gamma(0)\in r\overline\B^n$ and $\gamma(1)\in \C^n\setminus s\B^n$ 
has Euclidean length $\ge M$.
\end{enumerate}
\end{lemma}

\begin{remark}\label{rem:labyrinth}
Given a Riemannian metric $\ggot$ on $\C^n$ and a constant $M>0$,
there is a labyrinth $\Gamma$ as in 
the lemma such that any path crossing the shell $s\B^n\setminus r\overline \B^n$ 
and avoiding $\Gamma$ has $\ggot$-length $\ge M$. The reason is that any two 
Riemannian metrics are comparable on a compact set. 
When condition (b) holds for a metric $\ggot$, we say that 
$\Gamma$ enlarges the $\ggot$-distance by $M$. 
\end{remark}

Granted the lemma, the proof in \cite{Globevnik2015AM}
proceeds as follows. The main case is to find a complete foliation by  
hypersurfaces given as level sets of a holomorphic function $f:\B^n\to\C$. 
(By Sard's theorem, most level sets of $f$ are nonsingular.)
Pick a sequence $0<r_1<r_2<\cdots <1$ with $\lim_{i\to\infty}{r_i}=1$.
In each spherical shell $S_i=\{z\in\C^n: r_i<|z|<r_{i+1}\}$ we choose a labyrinth
$\Gamma_i$ satisfying Lemma \ref{lem:labyrinth} with a constant $M_i>0$, 
chosen such that 
\begin{equation}\label{eq:Mi}
	\sum_{i=1}^\infty M_i =+\infty. 
\end{equation}
Since the compact set $K_i=r_i\overline \B^n\cup \Gamma_i$ is 
polynomially convex for every $i\in\N$, a standard construction using 
the Oka--Weil theorem gives a holomorphic function $f$ on $\B^n$ satisfying 
\begin{equation}\label{eq:minGammai}
	\lim_{i\to\infty} \min_{z\in \Gamma_i}|f(z)| =+\infty.
\end{equation}
It follows that any divergent path $\gamma:[0,1)\to \B^n$ on which 
$f$ is bounded avoids $\Gamma_i$ for all sufficiently big $i\in\N$,
and hence $\gamma$ has infinite Euclidean length by Lemma \ref{lem:labyrinth} (b)
and \eqref{eq:Mi}. Thus, every nonempty level set $\{f=c\}$ is  
a complete closed (possibly singular) complex hypersurface in $\B^n$.
Foliations of lower dimensions are obtained by adding generically 
chosen additional component functions, noting that the 
leaves are automatically complete. 

The construction in the proof of Theorem \ref{th:main} is similar to the one of Globevnik, 
except that we also use the results of Forstneri\v c \cite{Forstneric2003AM} on 
the existence of holomorphic submersions from Stein manifolds to Euclidean spaces.
His main result (see \cite[Theorem 2.5]{Forstneric2003AM}) 
is that a holomorphic submersion $f:X \to \C^q$ always exists if 
$1\le q\le \left[\frac{n+1}{2}\right]$ where $n=\dim X$,
so we obtain a nonsingular holomorphic foliation of $X$ 
of any dimension $k=n-q$ with $\left[\frac{n}{2}\right]\le k\le n-1$.
If on the other hand $\left[\frac{n+1}{2}\right]<q\le n-1$ then a holomorphic 
submersion $X^n\to\C^q$ exists if and only if the holomorphic tangent bundle 
of $X$ admits a surjective complex vector bundle map $\theta:TX\to X\times \C^q$
onto the trivial bundle of rank $q$. 
In fact, the h-principle holds: every surjective complex vector bundle 
map $\theta$ as above is homotopic through maps 
of the same kind to the tangent map of a holomorphic submersion $f:X\to\C^q$.
In \cite[Theorem 2.5]{Forstneric2003AM} it is also shown that 
the analogous results hold with interpolation on closed complex subvarieties 
and approximation on compact $\Oscr(X)$-convex subsets of $X$,
in analogy to the constructions of holomorphic functions in the Oka--Cartan 
theory and of holomorphic maps in Oka theory. 
Noncritical holomorphic functions also exist on reduced Stein spaces
with singularities, see \cite{Forstneric2016JEMS}.

%
% PROOF OF THEOREM 1.1
%
\begin{proof}[Proof of Theorem \ref{th:main}]
We embed the Stein manifold $X$ as a closed complex submanifold
in a Euclidean space $\C^N$; see \cite[Theorems 2.4.1 and 9.3.1]{Forstneric2017E}
for a survey of the classical embedding theorems. 
For every $i\in\N$ the set $X_i=X\cap i \overline \B^N$ is compact and 
$\Oscr(X)$-convex, i.e., holomorphically convex in $X$. Pick a number $c_i>0$ such that 
\begin{equation}\label{eq:estimate1}
	\text{$\ggot(x,v)\ge c_i |v|$\ \ holds for every $x\in X_{i+1}$ and $v\in T_x X$.}
\end{equation} 
Here, $|v|$ denotes the Euclidean length of a tangent vector 
$v\in T_xX\subset T_x\C^N\cong \C^N$. 
In every spherical shell $S_i=\{z\in\C^N: i<|z|<i+1\}$ $(i\in\N)$ 
we choose a labyrinth $\Gamma_i$ satisfying Lemma \ref{lem:labyrinth} with a 
constant $M_i>0$, chosen such that 
\begin{equation}\label{eq:ciMi}
	\sum_{i=1}^\infty c_i M_i =+\infty. 
\end{equation} 

We explain the basic case with $k=n-1$.
Using the Oka--Weil type approximation theorem for noncritical holomorphic functions
(see \cite[Theorem 2.1]{Forstneric2003AM}) inductively, we find a holomorphic
function $f:X\to\C$ without critical points such that condition \eqref{eq:minGammai}
holds for all $i\in \N$ with $\Gamma_i$ replaced by $\Gamma'_i:=\Gamma_i\cap X$.
To be precise, choose constants $C_i>0$ $(i\in \N)$ with $\lim_{i\to\infty} C_i=+\infty$.
In the induction step, we are given a noncritical holomorphic function $f_i$
on a neighbourhood of $X_i$ in $X$. Next, we define $f_i$ on a neighbourhood 
of $\Gamma'_{i}$ in $X$ to be an arbitrary noncritical holomorphic function 
satisfying $\min_{\Gamma'_i}|f_{i}| >C_i$. Since $X_i\cup \Gamma'_i$ is a compact 
$\Oscr(X)$-convex set, we can apply \cite[Theorem 2.1]{Forstneric2003AM} 
to find a noncritical holomorphic function $f_{i+1}$ on a neighbourhood 
of $X_{i+1}$ in $X$ which approximates $f_i$ on $X_i\cup \Gamma'_i$. 
Assuming that the approximation is close enough at every step, the sequence 
$f_i$ converges uniformly on compacts in $X$ to a noncritical holomorphic function
$f:X\to \C$ satisfying $\min_{\Gamma'_i} |f| > C_i$ for every $i\in\N$.
Pick a divergent path $\gamma:[0,1)\to X$ contained in a level set of $f$ 
and choose $i_0\in\N$ such that $\gamma(0)\in i_0\B^N$. 
Note that $\gamma$ also diverges in $\C^N$ since $X$ is a closed noncompact submanifold.
We see as before that there is an integer $i_1\ge i_0$ such that $\gamma$ avoids 
the labyrinth $\Gamma'_i$, and hence also the labyrinth $\Gamma_i\subset\C^N$
for all $i\ge i_1$. From \eqref{eq:estimate1} and \eqref{eq:ciMi} it follows that 
\[
	\int_0^1 \ggot(\gamma(t),\dot\gamma (t)) dt \ \ge\ 
	\sum_{i=i_1}^\infty c_iM_i=+\infty.
\]
Hence, every nonempty level set $f^{-1}(c)$ is a closed complete complex hypersurface 
in $X$, which is smooth since $f$ has no critical points.

The same argument gives nonsingular $\ggot$-complete holomorphic foliations of $X$ 
of any dimension $k\ge \left[\frac{n}{2}\right]$. In this case, one inductively uses 
the Oka-Weil type approximation theorem for holomorphic submersions
(see \cite[Theorem 2.5]{Forstneric2003AM}) in the above proof; we leave
out the obvious details. The same holds for 
$1\le k< \left[\frac{n}{2}\right]$ if we assume the existence of a 
surjective complex vector bundle map $TX\to X\times \C^q$
with $q=n-k > \left[\frac{n+1}{2}\right]$. 
\end{proof}

By using the full power of \cite[Theorem 2.5]{Forstneric2003AM} 
we obtain the following stronger statement. 
%See also \cite[Corollary 2.3]{Forstneric2003AM} for hypersurface foliations.

%
% MAIN THEOREM 2
%
\begin{theorem}\label{th:main2}
Let $X$ be a closed complex submanifold of dimension $n>1$
in $\C^N$ endowed with a Riemannian metric $\ggot$. 
For every integer $k$ with $N - \left[\frac{n+1}{2}\right] \le k \le N-1$ there is a 
nonsingular holomorphic foliation $\Fcal$ of dimension $k$ on $\C^N$, 
given by a holomorphic submersion $f:\C^N\to \C^{N-k}$, such that every 
leaf $\Fcal_z$ is transverse to $X$ and the induced foliation on $X$ is $\ggot$-complete. 
\end{theorem}

\begin{proof}
Theorem \ref{th:main} furnishes a holomorphic submersion 
$f_0:X\to \C^q$ with $\ggot$-complete leaves, where $q=N-k\le \left[\frac{n+1}{2}\right]$.
Clearly, $f_0$ extends to a holomorphic submersion $U\to \C^q$ from an open 
neighbourhood $U$ of $X$ in $\C^N$. By \cite[Theorem 2.5]{Forstneric2003AM} there 
is a holomorphic submersion $f:\C^N\to\C^q$ which agrees with $f_0$ to the 
second order along $X$. The foliation on $\C^N$ determined by $f$ then clearly satisfies 
Theorem \ref{th:main2}. Furthermore, given a Riemannian 
metric $\hgot$ on $\C^N$, one can choose $\Fcal$ to be $\hgot$-complete,
transverse to $X$, and such that its trace on $X$ is $\ggot$-complete.
We leave further details to the reader.
\end{proof}

%
%  A more precise statement in the spirit of Globevnik, Charpentier and Kosinski:
%
\begin{remark}
The proof of Theorem \ref{th:main} actually gives for any  $1\le q \le \left[\frac{n+1}{2}\right]$
(where $n=\dim X$)  a holomorphic submersion $f:X\to \C^q$ which is unbounded
on every divergent path of finite $\ggot$-length in $X$. The same holds also for 
$q\in \{\left[\frac{n+1}{2}\right]+1,\ldots,n-1\}$ if there is a surjective complex vector
bundle map $TX\to X\times\C^q$. This is in the spirit of Globevnik's main theorem
in \cite{Globevnik2015AM} when $X$ is the ball in $\C^n$ and $f$ is a single 
holomorphic function, possibly with critical points. Furthermore, one can 
strengthen the statement in the spirit of those obtained by 
Charpentier and Kosi\'{n}ski \cite{CharpentierKosinski2021} for 
holomorphic functions on pseudoconvex domains in $\C^n$. In particular,
under the above conditions, a holomorphic submersion $f:X\to \C^q$ can be chosen 
such that the image of any divergent path of finite $\ggot$-length in $X$ is 
everywhere dense in $\C^q$. It follows that the foliation of $X$ defined by $f$ 
is nonsingular and $\ggot$-complete. 
\end{remark}

%%%%%%%%%%%
%
%  SECTION 3	
%
%%%%%%%%%%%
\section{Proofs of Theorems \ref{th:proper1}, \ref{th:proper2}, 
and \ref{th:proper3}}\label{sec:proper}

We shall need the following result generalizing \cite[Lemma 2]{AlarconForstneric2020MZ}.

%
%   THE MAIN LEMMA
%
\begin{lemma}\label{lem:avoiding}
Let $B$ be a compact polynomially convex set in $\C^n$ $(n>1)$,
and let $\Gamma=\bigcup_{j=1}^m \Gamma_j \subset \C^n\setminus B$ 
be a union of finitely many pairwise disjoint compact convex sets $\Gamma_j$
such that the set $B \cup \Gamma$ is polynomially convex.
If $E$ is a closed subset of $\C^n$ with unbounded complement, 
then for any $\epsilon>0$ there exists an automorphism $\Theta\in\Aut(\C^n)$ such that 
\begin{enumerate}[\rm (a)]
\item $\Theta(E)\cap \Gamma=\varnothing$, and 
\smallskip
\item $|\Theta(z)-z|<\epsilon$ for all $z\in B$.
\end{enumerate}
\end{lemma}

\begin{proof}
Pick a compact neighbourhood $K_0$ of $B$ and compact convex neighbourhoods 
$K_j$ of $\Gamma_j$ for $j=1,\ldots,m$ such that the sets $K_0,\ldots,K_m$ are 
pairwise disjoint and $K=\bigcup_{j=0}^m K_j$ is polynomially convex. 
(The existence of such sets follows from standard results on polynomial
convexity; see Stout \cite{Stout2007}.)
Let $\Psi_0=\Id\in\Aut(\C^n)$ be the identity and set $K'_0=K_0$. For every $j=1,\ldots,m$ 
we choose an automorphism $\Psi_j\in \Aut(\C^n)$ such that the compact sets 
$K'_j:=\Psi_j(K_j)$ are pairwise disjoint, we have that
\begin{equation}\label{eq:disjoint1}
	K'_j \cap (E\cup K'_0) =\varnothing\quad \text{for $j=1,\ldots,m$},
\end{equation}
and the union $\bigcup_{j=0}^m K'_j$ is polynomially convex. 
Such $\Psi_j$ are obtained by squeezing $K_j$ $(j=1,\ldots,m)$ 
by a dilation into a small neighbourhood of an interior point, then  
translating the images into small pairwise disjoint balls around 
some points in $\C^n\setminus (E\cup K'_0)$ (this set is open and 
nonempty by the assumption on $E$), and applying
\cite[Theorem 1.1]{ForstnericRosay1993} to approximate the final
map by an automorphism of $\C^n$. By \cite[Theorem 2.3]{ForstnericRosay1993} 
(see also \cite[Corollary 4.12.4]{Forstneric2017E}), given $\delta>0$ 
there is an automorphism $\Psi\in\Aut(\C^n)$ satisfying 
\begin{equation}\label{eq:Psi}
	|\Psi(z)-\Psi_j(z)| < \delta\quad \text{for all $z\in K_j$, \ $j=0,1,\ldots,m$.}
\end{equation}
Let $\Theta=\Psi^{-1}$. If $\delta>0$ is small enough then condition (b)  
holds, and we have that $\Psi(\Gamma_j) \subset \mathring K'_j$ and hence 
$\Gamma_j\subset \Theta(\mathring K'_j)$ for every $j=1,\ldots,m$, 
which by \eqref{eq:disjoint1} also yields (a).  
\end{proof}

\begin{proof}[Proof of Theorem \ref{th:proper1}]
Let $\B$ denote the open unit ball of $\C^n$.
Set $r_1=1$, $B_1=r_1\overline\B=\overline\B$, $\Fcal_1=\Fcal$, and $E_1=\Fcal_1(B_1)$; 
see \eqref{eq:FK}. Choose a sequence $M_j>0$ such that 
\begin{equation}\label{eq:sumMj}
	\sum_{j=1}^\infty M_j=+\infty. 
\end{equation}
Pick a closed ball $B'_1\subset \C^n$ 
centred at $0$ and containing $B_1$ in its interior.
By Lemma \ref{lem:labyrinth} and Remark \ref{rem:labyrinth} there is 
a labyrinth $\Gamma_1\subset \mathring B'_1\setminus  B_1$ which enlarges 
the $\ggot$-distance by $M_1>0$ such that $B_1\cup \Gamma_1$
is polynomially convex. Fix a number $0<\epsilon_1<1$. 
Lemma \ref{lem:avoiding} furnishes an automorphism $\phi_1\in\Aut(\C^n)$ 
such that 
\begin{enumerate}
\item[\rm (a$_1$)] $\phi_1(E_1)\cap \Gamma_1=\varnothing$, and 
\smallskip
\item[\rm (b$_1$)] $|\phi_1(z)-z|<\epsilon_1$ for all $z\in B_1$.
\end{enumerate}
In the second step, we choose a number $r_2>2$ 
such that $\phi_1(\overline\B) \subset (r_2-1) \B$ and set 
\[
	B_2=r_2\overline \B, \quad 
	\Fcal_2=(\phi_1)_*\Fcal_1, \quad
	E_2=\Fcal_2(B_2). 
\]	
Pick a slightly bigger ball $B'_2\supset B_2$ containing $B_2$ in its interior.
By Lemma \ref{lem:labyrinth} there is a labyrinth 
$\Gamma_2 \subset \mathring B'_2\setminus B_2$ which enlarges 
the $\ggot$-distance by $M_2>0$. Given $\epsilon_2>0$, Lemma \ref{lem:avoiding} 
furnishes an automorphism $\phi_2\in\Aut(\C^n)$ such that 
\begin{enumerate}
\item[\rm (a$_2$)] $\phi_2(E_2)\cap \Gamma_2=\varnothing$, and 
\smallskip
\item[\rm (b$_2$)] $|\phi_2(z)-z|<\epsilon_2$ for all $z\in B_2$.
\end{enumerate}
Continuing inductively, we obtain an increasing sequence of numbers $r_j>0$, balls 
\begin{equation}\label{eq:balls}
	B_j=r_j\overline\B \subset B'_j \subset B_{j+1}=r_{j+1}\overline\B, 
\end{equation}
labyrinths $\Gamma_j\subset \mathring B_{j+1}\setminus B_j$, constants 
$\epsilon_j>0$, automorphisms $\phi_j\in\Aut(\C^n)$, and 
foliations $\Fcal_j$ of $\C^n$ such that, setting 
\begin{equation}\label{eq:notation}
	\text{$E_{j}=\Fcal_{j}(B_{j})$, \ \ 
	$\Phi_j=\phi_j\circ\cdots \circ\phi_1\in \Aut(\C^n)$, \ \ and\ \ 
	$\Fcal_{j+1}=(\Phi_j)_*\Fcal$}, 
\end{equation}
the following conditions hold for every for $j=1,2,\ldots$:
\begin{enumerate}[\rm (i$_j$)]
\item $\phi_j(E_j)\cap \Gamma_j=\varnothing$.
\item $|\phi_j(z)-z|<\epsilon_j$ for $z\in B_j$.
\item $0<\epsilon_{j+1} < \frac12 \min\{\epsilon_{j}, \dist(\phi_j(E_j),\Gamma_j)\}$.
(Note that the set $\phi_j(E_j)$ is closed and $\Gamma_j$ is compact, so these sets are 
at positive distance by (ii$_j$).)
\item $\Phi_j(j\overline \B) \cup B'_j \subset (r_{j+1}-1) \B$.
\item The labyrinth $\Gamma_j$ enlarges the $\ggot$-distance by $M_j$
and $B_j\cup\Gamma_j$ is polynomially convex.
\end{enumerate}
Assuming that we have obtained these quantities for indices $\le j$, the induction step goes
as follows. Let $\Fcal_{j+1}$ and $E_{j+1}$ be given by \eqref{eq:notation}.
Pick a number $\epsilon_{j+1}$ satisfying (iii$_j$) and 
a number $r_{j+1}>r_j$ satisfying (iv$_j$), and set $B_{j+1}=r_{j+1}\overline\B$. 
Choose a ball $B'_{j+1}\supsetneq B_{j+1}$. Lemma \ref{lem:labyrinth} gives a labyrinth 
$\Gamma_{j+1}\subset \mathring B'_{j+1}\setminus B_{j+1}$
satisfying (v$_{j+1}$). Then, Lemma \ref{lem:avoiding} gives an automorphism 
$\phi_{j+1}\in\Aut(\C^n)$ satisfying (i$_{j+1}$) and (ii$_{j+1}$), closing the induction.

We claim that the sequence $\Phi_j\in\Aut(\C^n)$ 
converges uniformly on compacts in $\C^n$ to an automorphism 
$\Phi\in \Aut(\C^n)$. Indeed, since 
\[
	|\phi_j(z)-z|< \epsilon_j < 1<\dist(B_j,\C^n\setminus B_{j+1})
	\ \ \text{for every $j=1,2,\ldots$}
\]
and $\sum_{j=1}^\infty \epsilon_j<\infty$
(see (ii$_j$)--(iv$_j$)), the sequence $\Phi_j$ converges uniformly on compacts in the 
domain $\Omega=\bigcup_{j=1}^\infty \Phi_j^{-1}(B_j)\subset\C^n$ to a biholomorphic 
map $\Phi:\Omega\to\C^n$ onto $\C^n$ (see \cite[Corollary 4.4.2]{Forstneric2017E}). 
Furthermore, condition (iv$_j$) ensures that $j\overline \B\subset \Omega$ for every
$j$, so $\Omega=\C^n$ and hence $\Phi\in\Aut(\C^n)$. 
Hence, the sequence of foliations $\Fcal_{j+1}=(\Phi_j)_*\Fcal$
converges uniformly on compacts in $\C^n$ to a limit foliation $\Gcal=\Phi_*(\Fcal)$.

It remains to show that the foliation $\Gcal$ is $\ggot$-complete. 
We must show that for any divergent path $\gamma:[0,1)\to \C^n$ 
contained in a leaf of $\Fcal$, the path $\tilde \gamma =\Phi\circ\gamma:[0,1)\to\C^n$ 
has infinite $\ggot$-length. 
(Since the foliation $\Fcal$ is assumed to have closed leaves of positive 
dimension, every divergent path in a leaf of $\Fcal$ is also divergent in $\C^n$. 
Note that $\tilde \gamma$ is a divergent path in $\C^n$ contained in a leaf of $\Gcal$.) 
For every $k\in\N$ let $\gamma_k=\Phi_k\circ \gamma:[0,1)\to \C^n$; this is a divergent
path contained in a leaf of the foliation $\Fcal_{k+1}=(\Phi_k)_*\Fcal$.
Pick $j\in\N$ such that $\gamma(0) \in j\B$. By (iv$_j$) we have that 
$\gamma_j(0)=\Phi_j(\gamma(0)) \in (r_{j+1}-1) \B\subset B_{j+1}$, and hence
$\gamma_j ([0,1))\subset E_{j+1}$ (see \eqref{eq:notation}). 
Conditions (i$_j$)--(iii$_j$) imply that for every $k>j$ we have that 
$\gamma_k(0)\in B_{j+1}=r_{j+1}\overline \B$ 
and the trace of $\gamma_k$ avoids the labyrinths $\Gamma_{i}$ for $i=j+1,\ldots,k$. 
Since $\gamma_k$ is a divergent path in $\C^n$, 
its $\ggot$-length is at least $\sum_{i=j+1}^k M_i$
by (i$_i$). As $k\to\infty$, it follows that $\gamma_k$ converges to a divergent path 
$\tilde \gamma =\Phi\circ\gamma:[0,1)\to\C^n$ which has infinite $\ggot$-length
in view of \eqref{eq:sumMj}.
\end{proof}

\begin{proof}[Proof of Theorem \ref{th:proper2}]
The proof will follow that of Theorem \ref{th:proper1} if we find 
labyrinths $\Gamma$ in the given Stein manifold $X$ which increase the length
of divergent paths avoiding $\Gamma$ by a given amount and whose pieces 
are holomorphically contractible in $X$. (Note that the labyrinths used 
in the proof of Theorem \ref{th:main}, which are obtained by intersecting  
labyrinths in $\C^N$ whose pieces are balls in affine real hyperplanes 
with the embedded submanifold $X\subset\C^N$,
need not satisfy this property.) To this end, we shall prove the following lemma
whose first part generalizes Lemma \ref{lem:labyrinth},
as well as \cite[Lemma 2.4]{CharpentierKosinski2020} by Charpentier and Kosi\'{n}ski, 
while the second part is an analogue of Lemma \ref{lem:avoiding} adjusted to this situation.
As before, $\B$ denotes the unit ball of $\C^N$.

%
% LABYRINTH
%
\begin{lemma}\label{lem:labyrinth2}
Let $X$ be a closed complex submanifold of $\C^N$, and let $\ggot$ be a 
Riemannian metric on $X$. Given numbers $0<r<s$ and $M>0$, there is a compact set 
$\Gamma\subset X\cap (s\B\setminus r\overline \B)$ satisfying the following conditions. 
\begin{enumerate}[\rm (a)]
\item The compact set $\Gamma\cup (X\cap (r \overline \B))$ is $\Oscr(X)$-convex.  
\item $\Gamma=\bigcup_{i=1}^m \Gamma_i$ is the union of finitely many 
pairwise disjoint compact sets $\Gamma_i$ such that every $\Gamma_i$ 
is convex in a certain local holomorphic chart on $X$. 
\item Every piecewise smooth path $\gamma:[0,1]\to X \setminus \Gamma$ 
with $\gamma(0)\in X\cap r\overline\B$ and $\gamma(1)\in X\setminus s\B$ 
has $\ggot$-length at least $M$.
\end{enumerate}
Given such $\Gamma$ and assuming in addition that $X$ has the density property,
then for every closed subset $E\subsetneq X$ whose complement
$X\setminus E$ is not relatively compact in $X$ and for any $\epsilon>0$ 
there exists an automorphism $\Theta\in\Aut(X)$ such that 
\begin{enumerate}[\rm (A)]
\item $\Theta(E)\cap \Gamma=\varnothing$, and 
\smallskip
\item $|\Theta(z)-z|<\epsilon$ for all $z\in X\cap \, r\overline\B$.
\end{enumerate}
\end{lemma}

\begin{proof}
We construct the labyrinth $\Gamma$ in two stages. 

In the first stage, we apply Lemma \ref{lem:labyrinth} to find a labyrinth 
$\Gamma^0=\bigcup_{j=1}^k \Gamma^0_j$ in the spherical shell 
$S_{r,s}=s\B\setminus  r\overline\B\subset \C^N$ which increases the
$\ggot$-length in $X$ by the given amount $M>0$ and 
the set $\Gamma^0 \cup r\overline \B$ is polynomially convex. 
Set $B_0=X\cap \, r\overline \B$. It follows that each of the compact sets 
$\Gamma^X_j=\Gamma^0_j\cap X$ ($j=1,\ldots,k$), 
$\Gamma^X= \bigcup_{j=1}^k\Gamma^X_j$, and 
$\Gamma^X \cup B_0$ are $\Oscr(X)$-convex. 
The construction of such labyrinths in \cite{AlarconGlobevnikLopez2019Crelle}
shown that the connected components $\Gamma^0_j$ of $\Gamma^0$ 
(which are closed balls in affine real hyperplanes in $\C^N$) may be chosen
with arbitrarily small diameter (this is an immediate consequence of Pythagoras theorem, 
see \cite[Lemma 2.3]{Alarcon2022JDG}); in particular we can choose them
small enough such that $\Gamma^X_j=\Gamma^0_j\cap X$ is contained in 
a holomorphic coordinate chart $U_j\subset X$ which is Runge in $X$ 
for every $j=1,\ldots,k$. (Most of the sets $\Gamma^X_j$ are empty, 
and we discard them from the above collection and adjust the indexes accordingly.) 
More precisely, for any $\Gamma^X_j\ne \varnothing$ we pick a point
$p_j\in \Gamma^X_j$ and let $\Sigma_j=T_{p_j}X \cong\C^n$ with $n=\dim X$
be the tangent plane of $X$ at $p_j$, with the orthogonal $\C$-linear projection
$\pi_j:\C^N\to \Sigma_j$. Then, we may assume that there is a Runge neighbourhood 
$U_j\subset X$ of $\Gamma^X_j$ such that the restricted projection 
\[
	\pi_j|_{U_j} : U_j \to \pi_j(U_j)=V_j \subset \Sigma_j\cong \C^n
\] 
is a biholomorphic map onto a ball $V_j\subset \Sigma_j$ around the point $p_j$. 
It follows that a compact set $K\subset U_j$ is $\Oscr(X)$-convex 
if and only if $\pi_j(K)$ is polynomially convex in $\Sigma_j\cong\C^n$. 

In the second stage, we choose for every $j=1,\ldots,k$ a compact $\Oscr(X)$-convex 
neighbourhood $B_j\subset U_j$ of $\Gamma^X_j$ such that the sets
$B_0=X\cap \, r\overline \B,B_1,\ldots, B_k$ are pairwise disjoint and 
$\bigcup_{j=0}^k B_j$ is $\Oscr(X)$-convex. 
(Note that every compact $\Oscr(X)$-set has a basis of compact $\Oscr(X)$-convex
neighbourhoods.) We now apply the construction of labyrinths with holomorphically 
contractible pieces in pseudoconvex domains in
\cite[Theorem 1.1]{CharpentierKosinski2020} by Charpentier and 
Kosi\'{n}ski, to find for each $j=1,\ldots,k$ 
a labyrinth $\Gamma_j \subset \mathring B_j\setminus \Gamma^X_j$ 
which is $\Oscr(B_j)$-convex (and hence $\Oscr(X)$-convex), its connected 
components are holomorphically contractible sets in $X$ 
(in fact, under the projection $\pi_j:U_j\to V_j\subset\Sigma_j$ 
they correspond to closed balls in affine real hyperplanes of $\Sigma_j\cong\C^n$), 
and the $\ggot$-distance in $X\setminus \Gamma_j$ from $\Gamma^X_j$ 
to $bB_j$ is at least $M$. In other words, any path in $U_j$ from $\Gamma^X_j$
to $bB_j$ which avoids $\Gamma_j$ has $\ggot$-length at least $M$.

We claim that the labyrinth $\Gamma=\bigcup_{j=1}^k \Gamma_j \subset X\cap S_{r,s}$ 
satisfies  conditions (a)--(c) in the lemma. The set $\Gamma_j$ is $\Oscr(B_j)$
convex for every $j=1,\ldots,k$. Since $\bigcup_{j=0}^k B_j$ is $\Oscr(X)$-convex,
it follows that $\Gamma\cup B_0=\Gamma\cup (X\cap r\overline \B)$ 
is $\Oscr(X)$-convex, so (a) holds. Condition (b) holds by the construction.
Let $\gamma$ be a path as in part (c) avoiding $\Gamma$. 
If $\gamma$ avoids the initial labyrinth $\Gamma^X=\Gamma^0\cap X$ then 
its $\ggot$-length is at least $M$ by the choice of $\Gamma^0$. If on the other hand
$\gamma$ intersects $\Gamma^X_j$ for some $j\in\{1,\ldots,k\}$, then 
$\gamma$ connects a point in $bB_j$ to $\Gamma^X_j$ avoiding the labyrinth
$\Gamma_j$, hence its $\ggot$-length is at least $M$ by the choice of $\Gamma_j$. 
This proves (c).

The second part of the lemma is obtained by following the proof of 
Lemma \ref{lem:avoiding}. The only point deserving an explanation is the 
construction of the automorphisms $\Psi_j$ $(j=1,\ldots,k)$ and $\Psi$ in the
proof of Lemma \ref{lem:avoiding}.  
Recall that the labyrinth $\Gamma_j=\bigcup_{i=1}^{k_j}\Gamma_{j,i}$
has holomorpically contractible $\Oscr(X)$-convex connected components $\Gamma_{j,i}$.
More precisely, there is a 1-parameter family of biholomorphic contractions 
on a pseudoconvex Runge neighbourhood of $\Gamma_{j,i}$ in $X$ 
which shrinks this set within itself almost to a point. 
We can then move the images of these small sets to $X\setminus (B_0\cup E)$ 
by an isotopy of biholomorphic maps through pseudoconvex Runge domains in $X$,
ensuring that the traces of these isotopies for $j=1,\ldots,k$ 
are pairwise disjoint, contained in $X\setminus B_0$, 
and their unions (together with $B_0)$ are Runge in $X$
for every value of the parameter. (The set $B_0$ remains fixed during this process.) 
Assuming that the Stein manifold $X$ has the density property, we can apply 
the approximation theorem for such isotopies of injective holomorphic maps
(see \cite[Theorem 4.10.5]{Forstneric2017E}) to obtain an automorphism 
$\Psi\in\Aut(X)$ as in the proof of Lemma \ref{lem:avoiding}.
The remainder of the proof is exactly as in the case $X=\C^n$. 
\end{proof}

With Lemma \ref{lem:labyrinth2} in hand, we obtain Theorem \ref{th:proper2}
by following the proof of Theorem \ref{th:proper1},
and we leave the obvious details to the reader.
\end{proof}

%
%   PROOF OF THEOREM 1.6
%
\begin{proof}[Proof of  Theorem \ref{th:proper3}]
Let $\Fcal_0$ be a proper holomorphic foliation on a Stein manifold $X$ 
with the density property, and let $\Omega$ be a pseudoconvex Runge 
domain in $X$. By using Lemma \ref{lem:labyrinth2} inductively, 
we find a normal exhaustion 
$\Omega_1\Subset \Omega_2\Subset\cdots\subset \bigcup_{i=1}^\infty\Omega_i=\Omega$
by open relatively compact pseudoconvex Runge domains and for each $i=1,2,\ldots$ a 
labyrinth $\Gamma_i\subset \Omega_{i+1}\setminus \overline \Omega_i$ 
having the properties (a)--(c) in Lemma \ref{lem:labyrinth2} for a given
constant $M_i>0$ chosen such that $\sum_i M_i=+\infty$. Furthermore,
we can ensure that the compact set $\Gamma_{i+1}\cup\overline \Omega_i$ 
is $\Oscr(\Omega)$-convex (and hence also $\Oscr(X)$-convex) for every $i\in\N$.
Since the components of the labyrinth $\Gamma=\bigcup_{i=1}^\infty \Gamma_i$
are holomorphically contractible, we can apply the proof of Theorem \ref{th:proper1}
to inductively twist the foliation $\Fcal_0$ by a sequence of automorphisms
$\Phi_i\in\Aut(X)$ $(i=1,2,\ldots)$, 
chosen so that they converge uniformly on compacts in $\Omega$,
and the leaves of the foliations  $\Fcal_i=(\Phi_i)_*\Fcal_{i-1}$
avoid more and more components of $\Gamma$ as $i\to\infty$. 
The limit holomorphic foliation $\Fcal=\lim_{i\to\infty}\Fcal_i$ 
on $\Omega$ is such that every leaf avoids all but finitely many labyrinths 
$\Gamma_i$, and hence it is $\ggot$-complete.
Furthermore, the construction ensures that every leaf is a pseudoconvex Runge 
domain in a leaf of the initial foliation $\Fcal_0$ on $X$. 
The details can be found (for the case $X=\C^n$)
in \cite{AlarconGlobevnikLopez2019Crelle} and \cite{AlarconForstneric2020MZ}. 
\end{proof}

%
%  MAKING A SUBMANIFOLD COMPLETE
%
In Theorems \ref{th:proper1}, \ref{th:proper2}, and \ref{th:proper3} we focused on 
constructing complete nonsingular holomorphic foliations. However, 
the proofs of these results also apply to individual 
closed complex submanifolds and yield the following result.

\begin{theorem}\label{th:single}
Let $X$ be a Stein manifold
% Antonio: I mentioned the condition on the dimension
% Franc: as before, it is superfluous.
% of dimension $>1$
with the density property, and let $Y$ be a closed complex submanifold of $X$. 
The following assertions hold.
\begin{enumerate}[\rm (i)]
\item Given a Riemannian metric $\ggot$ on $X$ there is a holomorphic automorphism $\Phi\in \Aut(X)$ such that the submanifold $\Phi(Y)$ is $\ggot$-complete.
\item Let $\Omega$ be a pseudoconvex Runge domain in $X$ such that 
$\Omega\cap Y\neq\varnothing$. Given a Riemannian metric $\ggot$ on $\Omega$
and a connected compact subset $K\subset \Omega\cap Y$, there is a $\ggot$-complete 
closed complex submanifold of $\Omega$ which is biholomorphic to a pseudoconvex 
Runge domain in $Y$ containing $K$.
\end{enumerate}
\end{theorem}

By combining assertion (i) in Theorem \ref{th:single} with the embedding theorems 
for Stein manifolds (see \cite{EliashbergGromov1992,Schurmann1997,
AndristForstnericRitterWold2016}) we obtain the following corollary.

\begin{corollary}
Every Stein manifold $Y$ of dimension $n\ge 1$ admits a
proper $\ggot$-complete holomorphic embedding in $(\C^N,\ggot)$
for any $N\ge \max\left\{3,\left[\frac{3n}{2}\right]+1\right\}$, and a proper 
$\ggot$-complete holomorphic embedding in $(X,\ggot)$ for any Riemannian
Stein manifold $X$ with the density property of dimension $\dim X\ge 2n+1$. 
\end{corollary}

%
%
% 	ACKNOWLEDGEMENTS
%
%
\medskip
\noindent {\bf Acknowledgements.} Alarc\'on is partially supported by the State Research Agency (AEI) via the grant no.\ PID2020-117868GB-I00 and the ``Maria de Maeztu'' Excellence Unit IMAG, reference CEX2020-001105-M, funded by MCIN/AEI/10.13039/501100011033/, Spain.

Forstneri\v c is supported by the 
European Union (ERC Advanced grant HPDR, 101053085) and grants P1-0291, J1-3005, and N1-0237 from ARRS, Republic of Slovenia.

We thank an anonymous referee for useful remarks, and the editor for the suggestion to make
the introduction more accessible to a wider audience.

%%%%%%%%%%
%%%%%%%%%%
%%%%%%%%%%
%%%%%%%%%%   THE BIBLIOGRAPHY
%%%%%%%%%%
%%%%%%%%%%

%{\bibliographystyle{abbrv} \bibliography{references}} 

\begin{thebibliography}{10}

\bibitem{Alarcon2022JDG}
A.~Alarc\'{o}n.
\newblock Complete complex hypersurfaces in the ball come in foliations.
\newblock {\em J. Differential Geom.}, 121(1):1--29, 2022.

\bibitem{Alarcon2022IUMJ}
A.~Alarc\'{o}n.
\newblock Wild holomorphic foliations of the ball.
\newblock {\em Indiana Univ. Math. J.}, 71(2):561--578, 2022.

\bibitem{Alarcon2022Yang}
A.~Alarc\'{o}n.
\newblock The {Y}ang problem for complete bounded complex submanifolds: a
  survey.
\newblock To appear in {\em Proceedings for the Biennial Conference of the Spanish Royal Mathematical Society 2022 (RSME Springer Series)}.
\newblock \url{https://arxiv.org/abs/2212.08521}.

\bibitem{AlarconForstneric2013MA}
A.~Alarc{\'o}n and F.~Forstneri\v{c}.
\newblock Every bordered {R}iemann surface is a complete proper curve in a
  ball.
\newblock {\em Math. Ann.}, 357(3):1049--1070, 2013.

\bibitem{AlarconForstneric2020MZ}
A.~Alarc{\'o}n and F.~Forstneri\v{c}.
\newblock A foliation of the ball by complete holomorphic discs.
\newblock {\em Math. Z.}, 296(1-2):169--174, 2020.

\bibitem{AlarconForstneric2021RMI}
A.~Alarc\'{o}n and F.~Forstneri\v{c}.
\newblock The {C}alabi-{Y}au problem for {R}iemann surfaces with finite genus
  and countably many ends.
\newblock {\em Rev. Mat. Iberoam.}, 37(4):1399--1412, 2021.

\bibitem{AlarconForstnericLopez2021}
A.~Alarc\'{o}n, F.~Forstneri\v{c}, and F.~J. L\'{o}pez.
\newblock {\em Minimal surfaces from a complex analytic viewpoint}.
\newblock Springer Monographs in Mathematics. Springer, Cham, 2021.

\bibitem{AlarconGlobevnik2017}
A.~Alarc{\'o}n and J.~Globevnik.
\newblock Complete embedded complex curves in the ball of {$\mathbb C^2$} can
  have any topology.
\newblock {\em Anal. PDE}, 10(8):1987--1999, 2017.

\bibitem{AlarconGlobevnikLopez2019Crelle}
A.~Alarc{\'o}n, J.~Globevnik, and F.~J. L{\'o}pez.
\newblock A construction of complete complex hypersurfaces in the ball with
  control on the topology.
\newblock {\em J. Reine Angew. Math.}, 751:289--308, 2019.

\bibitem{AlarconLopez2013MA}
A.~Alarc{\'o}n and F.~J. L{\'o}pez.
\newblock Null curves in {$\mathbb {C}^3$} and {C}alabi-{Y}au conjectures.
\newblock {\em Math. Ann.}, 355(2):429--455, 2013.

\bibitem{AndersenLempert1992}
E.~Anders{\'e}n and L.~Lempert.
\newblock On the group of holomorphic automorphisms of {${\mathbb C}^n$}.
\newblock {\em Invent. Math.}, 110(2):371--388, 1992.

\bibitem{AndristForstnericRitterWold2016}
R.~B. Andrist, F.~Forstneri\v{c}, T.~Ritter, and E.~F. Wold.
\newblock Proper holomorphic embeddings into {S}tein manifolds with the density
  property.
\newblock {\em J. Anal. Math.}, 130:135--150, 2016.

\bibitem{CharpentierKosinski2020}
S.~Charpentier and {\L}.~Kosi\'{n}ski.
\newblock Construction of labyrinths in pseudoconvex domains.
\newblock {\em Math. Z.}, 296(3-4):1021--1025, 2020.

\bibitem{CharpentierKosinski2021}
S.~Charpentier and {\L}.~Kosi{\'n}ski.
\newblock Wild boundary behaviour of holomorphic functions in domains of
  {{\(\mathbb{C}^N\)}}.
\newblock {\em Indiana Univ. Math. J.}, 70(6):2351--2367, 2021.

\bibitem{EliashbergGromov1992}
Y.~Eliashberg and M.~Gromov.
\newblock Embeddings of {S}tein manifolds of dimension {$n$} into the affine
  space of dimension {$3n/2+1$}.
\newblock {\em Ann. of Math. (2)}, 136(1):123--135, 1992.

\bibitem{Forstneric2003AM}
F.~Forstneri\v{c}.
\newblock Noncritical holomorphic functions on {S}tein manifolds.
\newblock {\em Acta Math.}, 191(2):143--189, 2003.

\bibitem{Forstneric2016JEMS}
F.~Forstneri\v{c}.
\newblock Noncritical holomorphic functions on {S}tein spaces.
\newblock {\em J. Eur. Math. Soc. (JEMS)}, 18(11):2511--2543, 2016.

\bibitem{Forstneric2017E}
F.~Forstneri\v{c}.
\newblock {\em Stein manifolds and holomorphic mappings (The homotopy principle
  in complex analysis)}, volume~56 of {\em Ergebnisse der Mathematik und ihrer
  Grenzgebiete. 3. Folge}.
\newblock Springer, Cham, second edition, 2017.

\bibitem{ForstnericKutzschebauch2022}
F.~Forstneri\v{c} and F.~Kutzschebauch.
\newblock The first thirty years of {Anders{\'e}n}-{Lempert} theory.
\newblock {\em Anal. Math.}, 48(2):489--544, 2022.

\bibitem{ForstnericRosay1993}
F.~Forstneri\v{c} and J.-P. Rosay.
\newblock Approximation of biholomorphic mappings by automorphisms of
  {${\mathbb C}^n$}.
\newblock {\em Invent. Math.}, 112(2):323--349, 1993.
\newblock {Erratum: {\it Invent. Math.}, 118(3):573--574, 1994}.

\bibitem{Globevnik2015AM}
J.~Globevnik.
\newblock A complete complex hypersurface in the ball of {${\mathbb C}^N$}.
\newblock {\em Ann. of Math. (2)}, 182(3):1067--1091, 2015.

\bibitem{Globevnik2016MA}
J.~Globevnik.
\newblock Holomorphic functions unbounded on curves of finite length.
\newblock {\em Math. Ann.}, 364(3-4):1343--1359, 2016.

\bibitem{Jones1979}
P.~W. Jones.
\newblock A complete bounded complex submanifold of {${\mathbb C}^{3}$}.
\newblock {\em Proc. Amer. Math. Soc.}, 76(2):305--306, 1979.

\bibitem{Kutzschebauch2020}
F.~Kutzschebauch.
\newblock Manifolds with infinite dimensional group of holomorphic
  automorphisms and the linearization problem.
\newblock In {\em Handbook of group actions. {V}}, volume~48 of {\em Adv. Lect.
  Math. (ALM)}, pages 257--300. Int. Press, Somerville, MA, [2020] \copyright
  2020.

\bibitem{MartinUmeharaYamada2009PAMS}
F.~Mart\'{i}n, M.~Umehara, and K.~Yamada.
\newblock Complete bounded holomorphic curves immersed in {{\(\mathbb {C}^2\)}}
  with arbitrary genus.
\newblock {\em Proc. Am. Math. Soc.}, 137(10):3437--3450, 2009.

\bibitem{Scardua2021}
B.~Sc{\'a}rdua.
\newblock {\em Holomorphic foliations with singularities. {Key} concepts and
  modern results}.
\newblock Lat. Am. Math. Ser. Cham: Springer; Instituto de Matem{\'a}tica y
  Ciencias Afines (IMCA), 2021.

\bibitem{Schurmann1997}
J.~Sch{\"u}rmann.
\newblock Embeddings of {S}tein spaces into affine spaces of minimal dimension.
\newblock {\em Math. Ann.}, 307(3):381--399, 1997.

\bibitem{Stout2007}
E.~L. Stout.
\newblock {\em Polynomial convexity}, volume 261 of {\em Progress in
  Mathematics}.
\newblock Birkh\"auser Boston, Inc., Boston, MA, 2007.

\bibitem{Varolin2001}
D.~Varolin.
\newblock The density property for complex manifolds and geometric structures.
\newblock {\em J. Geom. Anal.}, 11(1):135--160, 2001.

\bibitem{Yang1977AMS}
P.~Yang.
\newblock Curvature of complex submanifolds of {$C^{n}$}.
\newblock In {\em Several complex variables ({P}roc. {S}ympos. {P}ure {M}ath.,
  {V}ol. {XXX}, {P}art 2, {W}illiams {C}oll., {W}illiamstown, {M}ass., 1975)},
  pages 135--137. Amer. Math. Soc., Providence, R.I., 1977.

\bibitem{Yang1977JDG}
P.~Yang.
\newblock Curvatures of complex submanifolds of {${\mathbb C}^{n}$}.
\newblock {\em J. Differential Geom.}, 12(4):499--511 (1978), 1977.

\end{thebibliography}
%\begin{comment}

%\end{comment}

\end{document}